\documentclass[10pt, twoside, leqno]{article}

\usepackage{amsmath,amsthm}
\usepackage{amssymb}
\usepackage{amsmath}
\usepackage{amsfonts}
\usepackage{amssymb}
\usepackage{dsfont,bbm}

\usepackage{enumerate}

\usepackage[numbers]{natbib}

\newtheorem{thm}{Theorem}[section]
\newtheorem{conj}[thm]{Conjecture}
\newtheorem{cor}[thm]{Corollary}
\newtheorem{lemma}[thm]{Lemma}
\newtheorem*{lemmaa}{Lemma}

\theoremstyle{definition}

\newtheorem{remark}[thm]{Remark}

\numberwithin{equation}{section}

\newcommand{\be}{\begin{equation}}
\newcommand{\ee}{\end{equation}}

\newcommand{\E}{{\mathbb{E}}}

\newcommand{\R}{\mathbb{R}}

\newcommand\inner[2]{|\langle #1, #2 \rangle|}
\newcommand\inneri[2]{\langle #1, #2 \rangle}
\newcommand\norm[1]{\|#1\|_2}

\newcommand{\NN}{N^{-\frac 2{n-1} }}
\newcommand{\NNN}{N^{\frac 2{n-1} }}

\begin{document}


\baselineskip=17pt


\title{Approximation of the Euclidean ball by polytopes with a restricted number of facets}

\author{Gil Kur \\
Weizmann Institute of Science\\ 
 Rehovot, Israel\\
E-mail: gilkur1990@gmail.com}
\date{}

\maketitle


\renewcommand{\thefootnote}{}



\renewcommand{\thefootnote}{\arabic{footnote}}
\setcounter{footnote}{0}


	\begin{abstract}
	We prove that there is an absolute constant $ C$ such that for every $ n \geq 2 $ and $ N\geq 10^n, $ there exists a polytope $ P_{n,N} $ in $ \R^n $ with at most $ N $ facets that satisfies
		\begin{equation*}
		\Delta_{v}(D_n,P_{n,N}):=\text{vol}_n\left(D_n \Delta P_{n,N}\right)\leq C\NN \text{vol}_n\left(D_n\right)
		\end{equation*}
		and
		\begin{align*}
				 \Delta_{s}(D_n,P_{n,N})&:=\text{vol}_{n-1}\left(\partial\left(D_n\cup P_{n,N}\right)\right) - \text{vol}_{n-1}\left(\partial\left(D_n\cap P_{n,N}\right)\right)\\&\leq 4C\NN \text{vol}_{n-1}\left(\partial D_n\right),
		\end{align*} 
		where $ D_n $ is the $ n$-dimensional Euclidean unit ball.
	This result closes gaps from several papers of Hoehner, Ludwig, Sch\"utt and Werner. The upper bounds are optimal up to absolute constants. This result shows that a polytope with an exponential number of facets can approximate the $ n$-dimensional Euclidean ball with respect to the aforementioned distances.  
\end{abstract}
\section{Introduction}

Let $ K$  be a convex body in $ \R^n $ with $ C^2 $ boundary $ \partial K$ and everywhere positive Gaussian curvature $\kappa$. First, in \cite{gruber1993asymptotic} it was shown that
\begin{align*}
&\lim_{N\to\infty}\frac{\min\{\text{vol}_n\left(P \setminus K\right)\;|\textrm{$ K \subset P  $ and $P$ is a polytope with at most $N$ facets}\}}{\NN} \\&= \frac{1}{2}\text{div}_{n-1}\left(\int_{\partial K} \kappa\left(x\right)^{\frac{1}{n+1}}d\mu_{\partial K}\left(x\right)\right)^{\frac{n+1}{n-1}},
\end{align*}
where  $ \mu_{\partial K} $ denotes the surface measure of $ \partial K  $ and $\text{div}_{n-1}$ is a constant that depends only on the dimension. In \cite{zador1982asymptotic}, Zador proved  that $ \text{div}_{n-1} = (2\pi e)^{-1}n + o(n).$  Later, Ludwig \citep{ludwig1999asymptotic} showed a similar formula for arbitrarily position polytopes, namely  
\begin{align*}
&\lim_{N\to\infty}\frac{\min\{\Delta_{v}(P,D_n)\;|\, P\textrm{ is a polytope with at most }N \textrm{ facets}\}}{\NN} =\\& \frac{1}{2}\textrm{ldiv}_{n-1}\left(\int_{\partial K} \kappa\left(x\right)^{\frac{1}{n+1}}d\mu_{\partial K}\left(x\right)\right)^{\frac{n+1}{n-1}},
\end{align*}
where  $ \textrm{ldiv}_{n-1} $ is a positive constant that depends only on the dimension. In \cite{Lud06}, it was shown that $ \textrm{ldiv}_{n-1} \geq c.$  Specifically, they proved that  for every polytope $ P $ in $ \R^n $ with $ N \geq 10^n $ facets  
\begin{align}
\Delta_v(D_n, P) \geq c\NN\text{vol}_n\left(D_n\right) .
\end{align} 
For more details, please see Theorem 2 in \cite{Lud06}. 

The estimate for $ \textrm{div}_{n-1} $ implies that $\textrm{ldiv}_{n-1} \leq c_2n,$ which until this paper, was the best-known upper bound for $ \textrm{ldiv}_{n-1} $. Clearly, there is a gap of a factor of a dimension between the upper and lower bounds for $\textrm{ldiv}_{n-1}$. In this paper, we prove that removing the circumscribed restriction improves the order of approximation by a factor of dimension. Specifically, we show that for all $ N \geq 10^n$ there is a polytope $ P_{n,N} $ in $ \R^n $ with at most $ N $ facets, which is generated from a random construction, that satisfies
\begin{align}\label{firstineq}
\Delta_v(D_n, P_{n,N}) \leq C_n\NN\text{vol}_n\left(D_n\right), 
\end{align} 
where $ C_n$ is a positive constant that depends only on the dimension and is bounded by an absolute constant. A corollary of this result is that  $ \textrm{ldiv}_{n-1} \leq C,$ which closes the aforementioned gap in the estimates for $\textrm{ldiv}_{n-1}$ from \cite{ludwig1999asymptotic,Lud06}. 
This inequality also shows that one can approximate the $ n$-dimensional Euclidean ball in the symmetric volume difference by an arbitrarily positioned polytope with an exponential number of facets. This phenomena holds for the Hausdorff metric and the Banach-Mazur distance; see \cite{artstein2015asymptotic,aubrun2017alice}. 

When $ N $ is large enough, we improve the bound from Eq. \eqref{firstineq} to
\[
	\Delta_v(D_n, P_{n,N}) \leq \left(\int_{0}^{1}t^{-1}(1-e^{-\ln(2)t})dt + \int_{0}^{\infty}e^{-\ln(2)e^{t}}dt+O(n^{-0.5})\right)\NN\text{vol}_n\left(D_n\right),
\]   
which implies that  
\[
\text{ldiv}_{n-1} \leq (\pi e)^{-1}\left(\int_{0}^{1}t^{-1}(1-e^{-\ln(2)t})dt+\int_{0}^{\infty}e^{-\ln(2)e^{t}}dt\right) + o(1) \sim \frac{0.96}{\pi e}+o(1).
\]  
We also optimize the argument of Theorem 2 in \cite{Lud06} and prove that $ \text{ldiv}_{n-1} \geq (4\pi e)^{-1}+o(1).$ 

Recently, Hoehner, Sch\"utt and Werner \cite{Wer15}  considered a polytopal approximation of the ball with respect to the  surface area deviation, which is defined for any two compact sets $ A,B \subset\R^n$ with measurable boundary as follows: 
\[
	\Delta_s\left(A,\ B\right) := \text{vol}_{n-1}\left(\partial\left(A\cup B\right)\right) - \text{vol}_{n-1}\left(\partial\left(A\cap B\right)\right).
\]
It was also shown that for all polytopes $ Q $ in $ \R^n $ with $ N \geq {M}_n$ facets,
\[
\Delta_s\left(Q,D_n\right) \geq c_1\NN\text{vol}_{n-1}\left(\partial D_n\right)
\]
where ${M}_n$ is a natural number that depends only on the dimension $ n $ and  $ c_1$ is a positive absolute constant.  
We show that this bound is optimal up to an absolute constant, by using the aforementioned random construction to find a polytope  $ Q_{n,N}$ in $ \R^n $ with at most $ N \geq 10^n $ facets that satisfies \[ \Delta_s\left(Q_{n,N},D_n\right) \leq 4C_n\NN\text{vol}_{n-1}\left(\partial D_n\right),\]
where $ C_n \leq C$ are the constants that were defined in Eq. \eqref{firstineq}.

		\paragraph{Notations and Preliminary Results}\ \\
		$ D_n$ is the $n\textrm{-dimensional}$ centered  Euclidean unit ball.$ |A| $ is the Lebesgue measure,i.e. volume, of a set $A.$ Similarly $ |\partial A| $ is the surface area of the set $ A $. $ \text{conv}(A) $ denotes the the convex hull of the set $ A.$ $ A^c $ denotes the complementary set of $ A.$  \\
		The symmetric volume difference between two sets $ |A \Delta B|$ is denoted by $ \Delta_v(A,B) $.\\
		The surface area deviation $  \Delta_s \left(A,B\right)$ := $ |\partial\left(A\cup B\right)| - |\partial\left(A\cap B\right)| .$\\
		We denote by $ \text{as}(K) := \int_{\partial K} \kappa\left(x\right)^{\frac{1}{n+1}}d\mu_{\partial K}\left(x\right) $ the affine surface area of $ C^2 $ convex body $ K,$
		and by $\sigma$ the uniform measure on the $ \mathbb{S}^{n-1}.$
Throughout the paper $ c,c',C,C',c_1,c_2,C_1,C_2$  denote positive absolute constants that may change from
line to line.
We shall use the following auxiliary results.
\begin{lemma}
	\[
	\frac{c}{\sqrt{n}} \leq \frac{|D_n|}{|D_{n-1}|} \leq  \frac{C}{\sqrt{n}}
	\]
\end{lemma}
\begin{thm}[Isoperimetric inequality] {\label{isoperemetric}}
	If $ K \subset \R^n $ be a convex body, then 
	\[
		|\partial K| \geq n|K|^{\frac{n-1}{n}}|D_n|^{\frac{1}{n}}.
	\]   
\end{thm}
\begin{thm}[Affine isoperimetric inequality \cite{lutwak1996brunn}]{\label{affine}}
	Let $ K \subset \R^n $ be a convex body with  $ |K|=|D_n| $ and let $ \text{as}(K) := \int_{\partial K} \kappa\left(x\right)^{\frac{1}{n+1}}d\mu_{\partial K}\left(x\right) $ denote the affine surface area of $K$. Then
	\[
		\textrm{as}(K)\leq \textrm{as}(D_n). 	
	\]
\end{thm}
\begin{thm}[Theorem 1 in \cite{ludwig1999asymptotic}]\label{ludlemma}
	\begin{align*}
		&\lim_{N\to\infty}\frac{\min\{\Delta_v(D_n,P_{n,N})\;|\textrm{P is a polytope with at most N facets}\}}{\NN} =\\& \frac{1}{2}\textrm{ldiv}_{n-1}\left(\int_{\partial K} \kappa\left(x\right)^{\frac{1}{n+1}}d\mu_{\partial K}\left(x\right)\right)^{\frac{n+1}{n-1}}
	\end{align*}
\end{thm}
\begin{thm}[Theorem 2 in \cite{Lud06}]{\label{lowebound}}
	Assume that $ N>10^{n} $, and let $ P $ be a polytope in $\R^n$ with at most $ N $ facets. Then there exists $ c>0 $ such that
	\[
	 \Delta_v(D_n,P_{n,N}) \geq c\NN|D_n|.
	\]
\end{thm}
	\section{Main results}
\begin{thm}{\label{main_thm}}
	Let $ P^b_{n,N} $  be  the polytope with at most $N$ facets that is best-approximating for $D_n$ with respect to the symmetric volume difference. Then for all $ N \geq n^n $, 
	\be
	\Delta_v(D_n,P^b_{n,N}) \leq \left(I+II+O\left(n^{-0.5}\right)\right)N^{-\frac{2}{n-1}}|D_n|,
	\ee
	where $ I = \int_{0}^{1}t^{-1}(1-e^{-\ln(2)t})dt$ and $ II= \int_{0}^{\infty}e^{-\ln(2)e^{t}}dt.$  It follows that
	\[
		\text{ldiv}_{n-1} \leq (\pi e)^{-1}(I+II) + o(1) \sim \frac{0.96}{\pi e} + o(1).
	\]
\end{thm}
\begin{remark}{\label{niceremark}}
	The bound on $ N$ can be improved from $ N \geq n^n $ to $ N \geq 10^{n}.$ This causes a change to the constant before $ \NN $. The proof is slightly different from the proof of Theorem \ref{main_thm}, and for completeness we provide a sketch of the proof in Section \ref{Techandloose}.
\end{remark} 
In \cite{Lud06}, it was shown that $ \Delta_v(D_n,P) \geq c\NN|D_n|.$ We optimize their argument to obtain the following result.
\begin{thm}{\label{suprisig}}
	 Let polytope $ P $ in $\R^n$ with at most $ N \geq n^n$ facets satisfies
	\[
	\Delta_v(D_n,P) \geq (\frac{1}{4}+O(\NN))N^{-\frac{2}{n-1}}|D_n|,
	\]
	and therefore $ \text{ldiv}_{n-1} \geq (4\pi e)^{-1} + o\left(1\right)$.  
\end{thm}

\begin{thm}{\label{sec_thm}}
	Let $ Q^b_{n,N} $ be the polytope with at most $N$ facets that is best-approximating for $D_n$ with respect to the surface area deviation. Then for all $ N \geq n^n $
	\be
	\Delta_s\left(Q^b_{n,N}, D_n\right) \leq \left(2\cdot I+II + \frac{1}{2} + O\left(n^{-0.5}\right)\right)\NN|\partial D_n|,
	\ee
	where $ I = \int_{0}^{1}t^{-1}(1-e^{-\ln(2)t})dt,$ $ II= \int_{0}^{\infty}e^{-\ln(2)e^{t}}dt.$
\end{thm}
	\begin{remark}
	The proof of Theorem \ref{sec_thm} implies that when $N\geq n^n$, there is a polytope $ P_{n,N} $ in $ \R^n $ with at most $ N$ facets that satisfies both
		\[
	\Delta_v\left(P_{n,N}, D_n\right) \leq \left(I+II+ O\left(n^{-0.5}\right)\right)N^{-\frac{2}{n-1}}|D_n|
	\]
	and
		\[
	\Delta_s\left(P_{n,N}, D_n\right) \leq \left(2\cdot I+II+ \frac{1}{2} + O\left(n^{-0.5}\right)\right)\NN|\partial D_n|,
	\]
where $ I = \int_{0}^{1}t^{-1}(1-e^{-\ln(2)t})dt$ and $ II= \int_{0}^{\infty}e^{-\ln(2)e^{t}}dt.$
	\end{remark}

	\begin{remark}
	In Theorem $ \ref{sec_thm}$ the bound of the number of facets can be improved from $ N\geq n^n $ to $ N \geq 10^{n}.$ This causes a  change to the constant before $ \NN $.
\end{remark}
\begin{remark}
			The author conjectures that in Theorem \ref{sec_thm} the estimate for the constant before the $ \NN $ can be improved.  
\end{remark}
\subsection{Asymptotic results}
In this section, we present some asymptotic results. First, let $ P^b_{n,N}\subset \R^n$ be the polytope with at most $N$ facets that is best-approximating for $D_n$ with respect to the symmetric volume difference.
The following corollaries are consequences of Theorem  \ref{main_thm}, Lemma \ref{lowebound} and Remark \ref{niceremark}.    
\begin{cor}
	If $ A \geq 10$ and the dimension is large enough, then 
	\[
	 \frac{\Delta_v\left(D_n,P^b_{n,A^{n}}\right)}{|D_n|} \in [cA^{-2},CA^{-2}].
	\] We conjecture that the limit $ \lim_{n\to\infty} \frac{\Delta_v\left(D_n,P^b_{n,A^{n}}\right)}{|D_n|}  $ exists.
\end{cor}

\begin{cor}
	Let $ f\left(n\right) $ be a sequence that satisfies  $ f(n) = e^{\omega(n)}.$ Then 
	\[
	\lim_{n\to\infty} \frac{\Delta_v\left(D_n,P^b_{n,f(n)}\right)}{|D_n|} = 0.
	\]
\end{cor}
\begin{remark}{\label{noaprrox}}
	It can  be easily proven that if $ f\left(n\right) =e^{o(n)}$. Then,
	\[
	\lim_{n\to\infty} \frac{\Delta_v\left(D_n,P^b_{n,f(n)}\right)}{|D_n|} = 1.
	\] 
\end{remark}

\subsection{Conjectures}\label{asymptoticresutls}
Due to symmetry considerations, we believe that Remark \ref{noaprrox} can be strengthened to:
\begin{conj}
	If $ N \leq 2^n$ and the dimension is large enough, then
	\[
		\lim_{n\to\infty} \frac{\Delta_v(D_n,P^b_{n,N})}{|D_n|} = 1.
	\]
\end{conj} 
In order to present the last conjecture, we use a standard argument to show that if the dimension is fixed and the number of facets tends to infinity, then among all convex bodies with the same volume, the Euclidean ball is the hardest to approximate. 

For this purpose, let $ K $ be a convex body in $ \R^{n}$, and assume without loss of generality that $ |K| = |D_n|$. Then
\begin{equation}{\label{conjconj}}
\begin{aligned}\lim_{N\to\infty}\NNN\min_{P\text{ has at most \ensuremath{N} facets}}\Delta_{v}(K,P) & =\frac{1}{2}\textrm{ldiv}_{n-1}\text{as}(K)^{\frac{n+1}{n-1}}\\
& \leq\frac{1}{2}\textrm{ldiv}_{n-1}\text{as}(D_{n})^{\frac{n+1}{n-1}}\\
& =\lim_{N\to\infty}\NNN\min_{P\text{ has at most \ensuremath{N} facets}}\Delta_{v}(D_n,P),
\end{aligned}
\end{equation}
where the first and the last equalities follow from Lemma \ref{ludlemma}, and the  inequality follows from the affine isoperimetric inequality (Lemma \ref{affine}). The author believes that the limit in Eq. \eqref{conjconj} is unnecessary, i.e. 
\begin{conj}\label{macbeath1}
Fix $n\in\mathbb{N}, n\geq 2$ and $N\geq n+1$, and let $K$ be a convex body in $\R^n$. Then
	\[
	\min_{P \text{ has at most $ N $ facets}}\frac{\Delta_v(P,K)}{|K|} \leq \min_{P \text{ has at most $ N $ facets}}\frac{\Delta_v(P,D_n)}{|D_n|}.
	\]
\end{conj}
Observe that by Theorem \ref{main_thm} there is a polytope with  $ f(\varepsilon,n):=(c\varepsilon)^{-\frac{n-1}{2}}$ facets that gives an $\varepsilon$-approximation of the $ n $-dimensional Euclidean ball, i.e. $ \frac{\Delta_v(P_{n,f(\varepsilon,n)}, D_n)}{|D_n|} \leq \varepsilon.$  Lemma \ref{lowebound} then implies that this result is optimal, up to an absolute constant. If Conjecture \ref{macbeath1} holds, then it follows that all convex bodies can be approximated by polytopes with an exponential number of facets with respect to the symmetric volume difference. 
\begin{remark}
	Macbeath \cite{macbeath1951extremal} showed that if $n\geq 2$ and  $N\geq n+1$, then for  every convex body $ K $ in $ \R^n $  
	\[
		\min_{P \text{ has at most $ N $ vertices, $ P \subset K $}}\frac{\Delta_v(P,K)}{|K|} \leq \min_{P \text{ has at most $ N $ vertices, $ P \subset D_n $}}\frac{\Delta_v(P, D_n)}{|D_n|}.
	\]
\end{remark}
\section{Proofs}
	For the proofs of Theorems \ref{main_thm} and \ref{sec_thm} we may assume that $ N $ is even.
	We also denote by $ \sigma $ the uniform probability measure on $ \mathbb{S}^{n-1},$ and recall that $ N \geq n^n.$
	
\subsection{Proof of Theorem \ref{main_thm}}
First, choose a random $ y\in\ \mathbb{S}^{n-1}$ from the uniform distribution on the sphere, and define the random slab of width $t$ as the set $ \{x\in\R^n:\inner{x}{y}\leq t\} $.   
Then, the probability that a point $ x \in\R^n$ lies outside of a random slab with width $ t \in \left(0,1\right)$ equals
\begin{equation}{\label{alphanr}}
\begin{aligned}\sigma_{y\in\mathbb{S}^{n-1}}\left(\inner{x}{y}\geq t\right) & =\sigma_{y\in\mathbb{S}^{n-1}}\left(\inner{\frac{x}{\norm{x}}}{y}\geq\frac{t}{\norm{x}}\right)\\
& =\frac{|\text{conv}(\vec{0},\{y\in\mathbb{S}^{n-1}:\inner{\frac{x}{\norm{x}}}{y}\geq\frac{t}{\norm{x}}\}\cap\partial D_{n})|}{|D_{n}|}\\
& =\frac{|\text{conv}(\vec{0},\{y\in\mathbb{R}^{^{n}}:\inner{\frac{x}{\norm{x}}}{y}=\frac{t}{\norm{x}}\}\cap D_{n})|}{|D_{n}|}\\
&   +\frac{|\{y\in\mathbb{R}^{n}:\inner{\frac{x}{\norm{x}}}{y}\geq\frac{t}{\norm{x}}\}\cap D_{n}|}{|D_{n}|}\\
& =\frac{2|D_{n-1}|}{|D_{n}|}\left(\int_{\frac{t}{\|x\|_{2}}}^{1}\left(1-x^{2}\right)^{\frac{n-1}{2}}dx+\frac{t}{n\norm{x}}\left(1-\frac{t^{2}}{\|x\|_{2}^{2}}\right)^{\frac{n-1}{2}}\right),
\end{aligned}
\end{equation}
where the  first term is the volume of the spherical cap and the second is the volume of the cone with $ \vec{0} $ as its apex, and both sets have the common base  $ \{y\in\mathbb{R}^{n}:\inneri{\frac{x}{\norm{x}}}{y}=\frac{t}{\norm{x}}\} \cap D_n.$ 

For shorthand, we denote by $ r=\|x\|_2 $ and the probability $ \sigma_{y\in \mathbb{R}^{n-1}}\left( \inner{x}{y} \geq t\right) $ by $\alpha_{n,r,t}.$ 
Let $ P $ be the random polytope that is generated by the intersection of $ \frac N2 $ independent random slabs with the same width $ t $. Observe that with probability one, $ P $ is \textbf{bounded} and has $ N $ facets. 

By independence, the probability that a point $x\in\R^n$ lies inside the random polytope $P$ equals
\[ 
\Pr\left(x\in P\right) = \Pr_{y_1,\ldots,y_{\frac N2} \in \mathbb{S}^{n-1}}\left(\cap_{i=1}^{\frac N2} \inner{x}{y_i} \leq t\right) =\left(1-\alpha_{n,r,t}\right)^{\frac N2}.
\]
Using Fubini and polar coordinates, we express the expectation of the random variable $ |D_n \setminus P| $ as 
\begin{align*}\E[|D_{n}\setminus P |] & =\int_{\otimes_{i=1}^{\frac{N}{2}}\mathbb{S}^{n-1}}\int_{D_{n}}(1-\mathbbm{1}_{\{x\in\cap_{i=1}^{\frac{N}{2}}\inner{x}{y_{i}}\leq t\}})dx\,d\sigma\left(y_{1}\right)\ldots d\sigma(y_{\frac{N}{2}})\\
& =\int_{D_{n}}\int_{\otimes_{i=1}^{\frac{N}{2}}\mathbb{S}^{n-1}}(1-\mathbbm{1}_{\{x\in\cap_{i=1}^{\frac{N}{2}}\inner{x}{y_{i}}\leq t\}})d\sigma\left(y_{1}\right)\ldots d\sigma(y_{\frac{N}{2}})dx\\
& =\int_{D_{n}}\left(1-\alpha_{n,\norm{x},t}\right)^{\frac{N}{2}}dx=|\partial D_{n}|\int_{t}^{1}r^{n-1}\int_{\mathbb{S}^{n-1}}\left(1-\alpha_{n,r,t}\right)^{\frac{N}{2}}d\sigma dr\\
& =|\partial D_{n}|\int_{t}^{1}r^{n-1}\left(1-\alpha_{n,r,t}\right)^{\frac{N}{2}}dr.
\end{align*}
The expectation $ \E[|P \setminus D_n|] $ can be expressed similarly, and thus
\begin{equation}\label{eq:Main_eq}
\begin{aligned}
\E[\Delta_v(D_n,P)] &= \E[|D_n \setminus P|] + \E[|P \setminus D_n|]  \\& ={|\partial D_n|}\left(\int_{t}^{1}r^{n-1}\left(1 - \left(1-\alpha_{n,r,t}\right)^{\frac N2}\right)dr +
\int_{1}^{\infty}r^{n-1}\left(1-\alpha_{n,r,t}\right)^{\frac N2}dr\right).
\end{aligned}
\end{equation}
Now we set $ t=t_{n,N} $ to be  
\[t_{n,N}=\sqrt{1 - \left(\frac{\gamma|\partial D_n|}{N|D_{n-1}|}\right)^{\frac 2{n-1}}}\] where $ \gamma $ is a positive absolute constant that will be determined later.

From now on, we use the notation $ \alpha_{n,r}$ instead of 
$ \alpha_{n,r,t_{n,N}} $.
%
We split the the proof of Theorem \ref{main_thm} into two main lemmas that give upper bounds for the two terms in Eq. \eqref{eq:Main_eq}.
\begin{lemma}{\label{firstpartlem}}
	\be\label{Partone}
	\begin{aligned}
		\E[|D_n \setminus P_{n,N}|] &= |\partial D_n|\int_{t_{n,N}}^{1}r^{n-1}\left(1 - \left(1-\alpha_{n,r}\right)^\frac N2\right)dr\\ 
		&= \left(\int_{0}^{1}t^{-1}(1-e^{-\gamma t})dt+O\left(n^{-0.5}\right)\right)N^{-\frac{2}{n-1}}|D_n|.
	\end{aligned}
	\ee
\end{lemma}\begin{lemma}{\label{lemma_2}}
\begin{equation}\label{Parttwo}
\E[|P \setminus D_n|] = |\partial D_n|\int_{1}^{\infty}r^{n-1}\left(1-\alpha_{n,r}\right)^{\frac N2}dr =\left(\int_{0}^{\infty}e^{-\gamma e^{t}}dt+O\left(n^{-0.5}\right)\right)N^{-\frac{2}{n-1}}|D_n|.
\end{equation}
\end{lemma}
First we show that Theorem \ref{main_thm} follows from the two aforementioned lemmas, and then we prove them.

\paragraph{Proof of Theorem \ref{main_thm}}
 Lemmas \ref{firstpartlem} and \ref{lemma_2} give the  upper bound 
\begin{equation}\label{symbound}
\E[\Delta_v(P,D_n)] =  \left(\int_{0}^{1} t^{-1}(1-e^{-\gamma t})dt + \int_{0}^{\infty}e^{-\gamma e^{t}}dt + O(n^{-0.5})\right)\NN|D_n|.
\end{equation}
Now we optimize over $ \gamma \in (0,\infty) $ and derive that the minimum is achieved at $ \gamma = \ln(2).$ This follows from the fact that
\begin{equation*}
\frac{\partial}{\partial\gamma}\left(\int_{0}^{1} t^{-1}(1-e^{-\gamma t})dt + \int_{0}^{\infty}e^{-\gamma e^{t}}dt + O(n^{-0.5})\right)=
\frac{1}{\gamma}(1-2e^{-\gamma}).
\end{equation*}
The main part of the theorem follows from the fact that there is polytope $ P_{n,N} $, a realization of $ P $, whose symmetric volume difference is no more than $\E[\Delta_v(P,D_n)] $. Finally, we give an upper bound for $ \text{ldiv}_{n-1} $. Observe that by Lemma \ref{ludlemma} and  Eq. \eqref{symbound},
\begin{align*} & \left(\int_{0}^{1}t^{-1}(1-e^{-\ln(2)t})dt+\int_{0}^{\infty}e^{-\ln(2)e^{t}}dt+O(n^{-0.5})\right)|D_{n}|\\
& \geq\frac{1}{2}\textrm{ldiv}_{n-1}\left(|\partial D_{n}|\right)^{\frac{n+1}{n-1}}=\frac{1}{2}(1+o(1))\textrm{ldiv}_{n-1}\frac{2\pi e}{n}|\partial D_{n}|\\
& =(1+o(1))\textrm{ldiv}_{n-1}\pi e|D_{n}|,
\end{align*}
and hence
\[
\text{ldiv}_{n-1} \leq (\pi e)^{-1}\left(\int_{0}^{1} t^{-1}(1-e^{-\ln(2)t})dt + \int_{0}^{\infty}e^{-\ln(2)e^{t}}dt+o(1)\right).
\]

\qed
\\

Now we turn our attention to the proofs of the main lemmas. We denote by $\delta = \left(n-1\right)^{-0.5}N^{-\frac{2}{n-1}}$, and we  use the following lemma which is proven in Section \ref{Techandloose}.
\begin{lemma}{\label{main_lemma}}
Let $ r \in [1-\frac{\NN}{\sqrt{n-1}}, 1+\frac{\NN}{\sqrt{n-1}}]$. Then
\begin{equation}
\alpha_{n,r}=\frac{2\gamma\left(1+O\left(n^{-1}\right)\right)}{N}e^{\left(n-1\right)\left(r-1\right)N^{\frac{2}{n-1}}\left(1+O\left(n^{-0.5}\right)\right)}.
\end{equation}
\end{lemma} 
\subsection{Proof of Lemma \ref{lemma_2}}
Let us split Eq. \eqref{Parttwo} into five parts:
\begin{equation}
\begin{aligned}|\partial D_{n}| & \bigg[\int_{1}^{1+\delta}r^{n-1}\left(1-\alpha_{n,r}\right)^{\frac{N}{2}}dr+\int_{1+\delta}^{1+2\NN}r^{n-1}\left(1-\alpha_{n,r}\right)^{\frac{N}{2}}dr\\
& \quad+\int_{1+2\NN}^{1+\frac{2}{n}}r^{n-1}\left(1-\alpha_{n,r}\right)^{\frac{N}{2}}dr+\int_{1+\frac{2}{n}}^{n^{2}}r^{n-1}\left(1-\alpha_{n,r}\right)^{\frac{N}{2}}dr\\
& \quad+\int_{n^{2}}^{\infty}r^{n-1}\left(1-\alpha_{n,r}\right)^{\frac{N}{2}}dr\bigg]
\end{aligned}
\end{equation}
Next, we estimate these integrals in a series of lemmas.
\begin{lemma}
\[
\int_{1}^{1+\delta}r^{n-1}\left(1-\alpha_{n,r}\right)^{\frac {N}{2}}dr = \left(\int_{0}^{\infty}e^{-\ln(2)e^{t}}dt+O\left(n^{-0.5}\right)\right)\NN|D_{n}|.
\]
\end{lemma}
\begin{proof}
By Lemma \ref{main_lemma}, if $r\in [1,1+\delta]$ then
\[
\alpha_{n,r}=\frac{2\gamma\left(1+O\left(n^{-1}\right)\right)}{N}e^{\left(n-1\right)\left(r-1\right)N^{\frac{2}{n-1}}\left(1+O\left(n^{-0.5}\right)\right)}.
\]
Hence,
\begin{equation}
\begin{aligned} & |\partial D_{n}|\int_{1}^{1+\delta}r^{n-1}\left(1-\alpha_{n,r}\right)^{\frac{N}{2}}dr=\\
& |\partial D_{n}|\int_{1}^{1+\delta}r^{n-1}\left(1-\frac{2\gamma\left(1+O\left(n^{-1}\right)\right)}{N}e^{\left(n-1\right)\left(r-1\right)N^{\frac{2}{n-1}}\left(1+O\left(n^{-0.5}\right)\right)}\right)^{\frac{N}{2}}dr=\\
& (1+O(n^{-1}))|\partial D_{n}|\int_{1}^{1+\delta}e^{-\left(1+O\left(n^{-1}\right)\right)\gamma e^{\left(n-1\right)\left(r-1\right)N^{\frac{2}{n-1}}\left(1+O\left(n^{-0.5}\right)\right)}}dr=\\
& \left(1+O\left(n^{-0.5}\right)\right)|D_{n}|\NN\int_{0}^{n^{0.5}}e^{-\gamma e^{t}}dt=\\
& \left(1+O\left(n^{-0.5}\right)\right)|D_{n}|\NN\int_{0}^{\infty}e^{-\gamma e^{t}}dt.
\end{aligned}
\end{equation}
\end{proof}
\begin{lemma}
\[
|\partial D_{n}|\int_{1+\delta}^{1+2\NN}r^{n-1}\left(1-\alpha_{n,r}\right)^{\frac{N}{2}}dr = |D_{n}|\NN o\left(n^{-0.5}\right).
\]
\end{lemma}
\begin{proof}
	Since  $ 1-\alpha_{n,r} $ is a decreasing function of $ r,$  we need to derive a \textbf{lower bound} for $ \alpha_{n,r}.$
First, by Lemma \ref{main_lemma} applied to $ r = {1+\delta}=1+\left(n-1\right)^{-0.5}N^{-\frac{2}{n-1}},$ we get that
\begin{align*}
\alpha_{n,{1+\delta}}&=\frac{2\gamma\left(1+O\left(n^{-1}\right)\right)}{N}e^{\left(n-1\right)\left(r-1\right)N^{\frac{2}{n-1}}\left(1+O\left(n^{-0.5}\right)\right)}\\&=\frac{2\gamma\left(1+O\left(n^{-1}\right)\right)}{N}e^{\sqrt{n-1}\left(1+O\left(n^{-0.5}\right)\right)}.
\end{align*}
Hence,
\begin{equation}
\begin{aligned} & |\partial D_{n}|\int_{1+\delta}^{1+2\NN}r^{n-1}\left(1-\alpha_{n,r}\right)^{\frac{N}{2}}dr\leq\\
& |\partial D_{n}|\int_{1+\delta}^{1+2\NN}r^{n-1}\left(1-\alpha_{n,1+\left(n-1\right)^{-0.5}N^{-\frac{2}{n-1}}}\right)^{\frac{N}{2}}dr=\\
& |\partial D_{n}|\int_{1+\delta}^{1+2\NN}r^{n-1}\left(1-\frac{2\gamma\left(1+O\left(n^{-1}\right)\right)}{N}e^{\sqrt{n-1}\left(1+O\left(n^{-0.5}\right)\right)}\right)^{\frac{N}{2}}dr\leq\\
& |\partial D_{n}|\int_{1+\delta}^{1+2\NN}r^{n-1}e^{-c\gamma e^{\sqrt{n-1}\left(1+O\left(n^{-0.5}\right)\right)}}dr\leq\\
& \left(1+o\left(n^{-1}\right)\right)|\partial D_{n}|e^{-c\gamma e^{\sqrt{n-1}\left(1+O\left(n^{-0.5}\right)\right)}}\int_{1}^{1+2\NN}r^{n-1}dr=\\
& |D_{n}|\NN o\left(n^{-0.5}\right).
\end{aligned}
\end{equation}
\end{proof}
\begin{lemma}
\[
|\partial D_{n}|\int_{1+2\NN}^{1+\frac{2}{n}}r^{n-1}\left(1-\alpha_{n,r}\right)^{\frac{N}{2}}dr = |D_{n}|N^{-\frac{2}{n-1}}o\left(n^{-0.5}\right).
\]
\end{lemma}
\begin{proof}
By Eq. \eqref{alphanr},
\begin{equation*}
\begin{aligned}
\alpha_{n,r} &\geq \frac{2|D_{n-1}|}{|D_n|}\frac {t_{n,N}}{nr} \left(1-\frac{t_{n,N}^2}{r^2}\right)^{\frac{n-1}{2}} \\&=\left(1+O\left(n^{-1}\right)\right)\frac{2|D_{n-1}|}{|\partial D_n|}\left(1-\frac{t_{n,N}^2}{r^2}\right)^{\frac{n-1}{2}}
\end{aligned}
\end{equation*}
where $t_{n,N}=\sqrt{1 - \left(\frac{\gamma|\partial D_n|}{N|D_{n-1}|}\right)^{\frac 2{n-1}}}$. Hence,
\begin{equation*}
\begin{aligned} & |\partial D_{n}|\int_{1+2\NN}^{1+\frac{2}{n}}r^{n-1}\left(1-\alpha_{n,r}\right)^{\frac{N}{2}}dr\leq\\
& e^{2}|\partial D_{n}|\int_{1+2\NN}^{1+\frac{2}{n}}\left(1-\alpha_{n,r}\right)^{\frac{N}{2}}dr\leq\\
& C|\partial D_{n}|\int_{1+2\NN}^{1+\frac{2}{n}}\left(1-\left(1+O\left(n^{-1}\right)\right)\frac{2|D_{n-1}|}{|\partial D_{n}|}\left(1-\frac{t_{n,N}^{2}}{r^{2}}\right)^{\frac{n-1}{2}}\right)^{\frac{N}{2}}dr\leq\\
& C|\partial D_{n}|\int_{1+2\NN}^{1+\frac{2}{n}}e^{-N\left(1+O\left(n^{-1}\right)\right)\frac{|D_{n-1}|}{|\partial D_{n}|}\left(1-\frac{t_{n,N}^{2}}{r^{2}}\right)^{\frac{n-1}{2}}}dr\leq\\
& C|\partial D_{n}|\int_{1+2\NN}^{1+\frac{2}{n}}e^{-N\left(1+O\left(n^{-1}\right)\right)r{}^{-\left(n-1\right)}\frac{|D_{n-1}|}{|\partial D_{n}|}\left(r^{2}-t_{n,N}^{2}\right)^{\frac{n-1}{2}}}dr.
\end{aligned}
\end{equation*}
Again, using the fact that $ r \leq 1 + \frac 2n, $ the previous expression is no more than
\begin{equation*}
\begin{aligned} & C|\partial D_{n}|\int_{1+2\NN}^{1+\frac{2}{n}}e^{-cN\frac{|D_{n-1}|}{|\partial D_{n}|}\left(\left(1+\left(r-1\right)\right)^{2}-t_{n,N}^{2}\right)^{\frac{n-1}{2}}}dr=\\
& C|\partial D_{n}|\int_{1+2\NN}^{1+\frac{2}{n}}e^{-cN\frac{|D_{n-1}|}{|\partial D_{n}|}\left(1-t_{n,N}^{2}+2\left(r-1\right)\right)^{\frac{n-1}{2}}}dr.
\end{aligned}
\end{equation*}
Now we use that $t_{n,N} = \sqrt{1-\left(\frac{\gamma|D_{n-1}|}{|\partial D_n|N}\right)^{\frac 2{n-1}}}  $ and the fact that $ \left(1+b\right)^n\geq 1+nb $ on $ [0,\infty) $ to derive that the previous expression equals
\begin{equation}
\begin{aligned} & C|\partial D_{n}|\int_{1+2\NN}^{1+\frac{2}{n}}e^{-c\gamma N\frac{|D_{n-1}|}{|\partial D_{n}|}\left(\left(\frac{\gamma|\partial D_{n}|}{|D_{n-1}|N}\right)^{\frac{2}{n-1}}+2\left(r-1\right)\right)^{\frac{n-1}{2}}}dr\leq\\
& C|\partial D_{n}|\int_{1+2\NN}^{1+\frac{2}{n}}e^{-2c\gamma\left(1+2\left(r-1\right)N^{\frac{2}{n-1}}\left(1+O\left(\frac{\ln\left(n\right)}{n}\right)\right)\right)^{\frac{n-1}{2}}}dr\leq\\
& C|\partial D_{n}|\int_{1+2\NN}^{1+\frac{2}{n}}e^{-2c\gamma\left(1+\left(n-1\right)\left(r-1\right)N^{\frac{2}{n-1}}\left(1+O\left(\frac{\ln\left(n\right)}{n}\right)\right)\right)}dr\leq\\
& C_{1}|\partial D_{n}|\int_{2\NN}^{\frac{2}{n}}e^{-c_{1}\gamma nN^{\frac{2}{n-1}}r}dr=|D_{n}|N^{-\frac{2}{n-1}}O\left(n^{-1}\right).
\end{aligned}
\end{equation}
\end{proof}
\begin{lemma}
\[
|\partial D_{n}|\int_{1+\frac{2}{n}}^{n^{2}}r^{n-1}\left(1-\alpha_{n,r}\right)^{\frac{N}{2}}dr = |D_{n}|N^{-\frac{2}{n-1}}o\left(n^{-0.5}\right)
\]
\end{lemma}
\begin{proof}
Recalling that $ \alpha_{n,r} $ is decreasing in $ r, $ we derive that
\begin{equation}\label{blabla}
\begin{aligned}
|\partial D_{n}|\int_{1+\frac{2}{n}}^{n^{2}}r^{n-1}\left(1-\alpha_{n,r}\right)^{\frac{N}{2}}dr&\leq
|\partial D_{n}|\int_{1+\frac{2}{n}}^{n^{2}}r^{n-1}\left(1-\alpha_{n,1+\frac 2n}\right)^{\frac{N}{2}}dr\\
&\leq|\partial D_{n}|n^{2n}\int_{1+\frac{2}{n}}^{n^{2}}\left(1-\alpha_{n,1+\frac 2n}\right)^{\frac{N}{2}}dr.
\end{aligned}
\end{equation}
In order to continue, we derive an upper bound for $ \alpha_{n,1+\frac 2n} .$ Using the fact that
\[
\alpha_{n,r} > \frac{2|D_{n-1}|}{|D_n|}\frac {t_{n,N}}{nr} \left(1-\frac{t_{n,N}^2}{r^2}\right)^{\frac{n-1}{2}},
\]
and also that $ t_{n,N} = 1-O\left(\frac 1{n^2}\right) $ and  $r = 1+\frac{2}{n},$ it holds that
\begin{equation*}
\begin{aligned}\alpha_{n,1+\frac{2}{n}} & \geq\left(1+O\left(n^{-1}\right)\right)\frac{2|D_{n-1}|}{|\partial D_{n}|}\left(1-\frac{t_{n,N}^{2}}{r^{2}}\right)^{\frac{n-1}{2}}\\
& \geq c_{1}n^{-0.5}\left(\frac{4}{n}+o\left(n^{-1}\right)\right)^{\frac{n-1}{2}}\geq c_{2}\left(\frac{4}{n}\right)^{\frac{n}{2}}.
\end{aligned}
\end{equation*}
Now we continue from the end of Eq. \eqref{blabla} to derive that
\begin{equation}
\begin{aligned}  |\partial D_{n}|n^{2n}\int_{1+\frac{2}{n}}^{n^{2}}\left(1-c_{2}\left(\frac{4}{n}\right)^{\frac{n}{2}}\right)^{\frac{N}{2}}dr&\leq\,|\partial D_{n}|n^{2n}\int_{1+\frac{2}{n}}^{n^{2}}e^{-Nn^{-\frac{n}{2}}}dr\\
&\leq  |\partial D_{n}|n^{2n+2}e^{-\sqrt{N}}\\
 &=|\partial D_{n}|n^{2n}\int_{1+\frac{2}{n}}^{n^{2}}e^{-\sqrt{N}}dr\\
 &=|D_{n}|N^{-\frac{2}{n-1}}o\left(n^{-0.5}\right),
\end{aligned}
\end{equation}
where we used the assumption that $ N\geq n^n.$ 
\end{proof}
The next lemma is proven in Section \ref{Techandloose} and will be used to prove Lemma \ref{3.9} below.
\begin{lemma}{\label{tech_lemma}}
	Assume that $ r\geq n^2 $. Then 
	\be
	\alpha_{n,r} \geq 1-\frac{C\sqrt{n}}{r}	.
	\ee
\end{lemma}

\begin{lemma}\label{3.9}
\[
|\partial D_{n}|\int_{n^{2}}^{\infty}r^{n-1}\left(1-\alpha_{n,r}\right)^{\frac{N}{2}}dr = |D_{n}|\NN o\left(n^{-0.5}\right)
\]
\end{lemma}
\begin{proof}
We have that
\begin{equation}
\begin{aligned}  |\partial D_{n}|\int_{n^{2}}^{\infty}r^{n-1}\left(1-\alpha_{n,r}\right)^{\frac{N}{2}}dr&\leq|\partial D_{n}|\int_{n^{2}}^{\infty}r^{n-1}\left(\frac{C\sqrt{n}}{r}\right)^{\frac{N}{2}}dr\\&\leq  |\partial D_{n}|C^{N}n^{\frac{N}{2}}\int_{n^{2}}^{\infty}r^{-\frac{N}{3}}dr\\&\leq|\partial D_{n}|C^{N}n^{\frac{N}{2}}n^{-\frac{2}{3}N+2}\int_{1}^{\infty}r^{-\frac{N}{3}}\\&=|D_{n}|\NN o\left(n^{-0.5}\right).
\end{aligned}
\end{equation}
\end{proof}
Putting everything together,  Lemma \ref{lemma_2} now follows from all of the lemmas that were proven in this subsection, and finally we derive that
\begin{equation}
\E[|P \setminus D_n|] = \left(\int_{0}^{\infty}e^{-\ln(2)e^{t}}dt+O\left(n^{-0.5}\right)\right)\NN|D_n|.
\end{equation}
\qed
\subsection{Proof of Lemma \ref{firstpartlem}}
	First, we split the integral of Eq. \eqref{Partone} into two parts
	\begin{align*}
	|\partial D_{n}|\left(\int_{1-\delta}^{1}r^{n-1}\left(1-\left(1-\alpha_{n,r}\right)^{\frac{N}{2}}\right)dr + \int_{t_{n,N}}^{1-\delta}r^{n-1}\left(1-\left(1-\alpha_{n,r}\right)^{\frac{N}{2}}\right)dr \right).
	\end{align*}
Next, we estimate the first integral.
	\begin{lemma}
		$$
		|\partial D_{n}|\int_{1-\delta}^{1}r^{n-1}\left(1-\left(1-\alpha_{n,r}\right)^{\frac{N}{2}}\right)dr = \left(\int_{0}^{1}t^{-1}(1-e^{-\gamma t})dt+O(n^{-0.5})\right)\NN|D_n|.
		$$
	\end{lemma}
	\begin{proof}
	For  $r\in [1-\delta,1] $, we use Lemma \ref{main_lemma} to estimate $ \alpha_{n,r} $  and derive that
	\[
	\alpha_{n,r}=\frac{2\gamma\left(1+O\left(n^{-1}\right)\right)}{N}e^{\left(n-1\right)\left(r-1\right)N^{\frac{2}{n-1}}\left(1+O\left(n^{-0.5}\right)\right)}.
	\]  
	Hence,
	\begin{equation*}
	\begin{aligned} & |\partial D_{n}|\int_{1-\delta}^{1}r^{n-1}\left(1-\left(1-\alpha_{n,r}\right)^{\frac{N}{2}}\right)dr\leq\\
	& |\partial D_{n}|\int_{1-\delta}^{1}r^{n-1}\left(1-\left(1-\frac{2\gamma\left(1+O\left(n^{-1}\right)\right)}{N}e^{\left(n-1\right)\left(r-1\right)N^{\frac{2}{n-1}}\left(1+O\left(n^{-0.5}\right)\right)}\right)^{\frac{N}{2}}\right)dr.
	\end{aligned}
	\end{equation*}
	Using the equality $ 1-x_n= \left(1+O\left(x_n^2\right)\right)e^{-x_n} $, where $ x_n = O\left(n^{-1}\right) $, we obtain 
	\begin{equation*}
	\begin{aligned} & |\partial D_{n}|\int_{1-\delta}^{1}r^{n-1}\left(1-e^{-\gamma e^{\left(1+O\left(n^{-1}\right)\right)\left(n-1\right)\left(r-1\right)N^{\frac{2}{n-1}}\left(1+O\left(n^{-0.5}\right)\right)}}\right)dr=\\
	& |\partial D_{n}|\int_{1-\delta}^{1}r^{n-1}\left(1-\left(1+O\left(n^{-1}\right)\right)e^{-\gamma e^{\left(1+O\left(n^{-1}\right)\right)\left(n-1\right)\left(r-1\right)N^{\frac{2}{n-1}}\left(1+O\left(n^{-0.5}\right)\right)}}\right)dr=\\
	& (1+O(n^{-1}))|\partial D_{n}|\int_{1-\delta}^{1}1-e^{-\gamma e^{\left(1+O\left(n^{-1}\right)\right)\left(n-1\right)\left(r-1\right)N^{\frac{2}{n-1}}\left(1+O\left(n^{-0.5}\right)\right)}}dr=\\
	& (1+O(n^{-0.5}))|D_{n}|\NN\int_{0}^{n^{0.5}}1-e^{-\gamma e^{-x}}dx=\\
	& (1+O(n^{-0.5}))|D_{n}|\NN\int_{0}^{1}t^{-1}(1-e^{-\gamma t})dt.
	\end{aligned}
	\end{equation*}
\end{proof}
We now estimate the second integral in Eq. \eqref{firstpartlem}.
\begin{lemma}
	\[
	|\partial D_n|\int_{t_{n,N}}^{1-\delta}r^{n-1}\left(1-\left(1-\alpha_{n,r}\right)^{\frac{N}{2}}\right)dr = o(n^{0.5})\NN|D_n|.
	\]
\end{lemma}
\begin{proof}
	Using Lemma \ref{main_lemma} with $ r = {1-\delta}=1-\left(n-1\right)^{-0.5}N^{-\frac{2}{n-1}},$ we get that
	\begin{align*}
	\alpha_{n,{1-\delta}}&=\frac{2\gamma\left(1+O\left(n^{-1}\right)\right)}{N}e^{\left(n-1\right)\left(r-1\right)N^{\frac{2}{n-1}}\left(1+O\left(n^{-0.5}\right)\right)}\\&=\frac{2\gamma\left(1+O\left(n^{-1}\right)\right)}{N}e^{-\sqrt{n-1}\left(1+O\left(n^{-0.5}\right)\right)}.
	\end{align*}
	Therefore,
	\begin{equation}\begin{aligned} & |\partial D_{n}|\int_{t_{n,N}}^{1-\delta}r^{n-1}\left(1-\left(1-\alpha_{n,r}\right)^{\frac{N}{2}}\right)dr\leq\\
	& |\partial D_{n}|\int_{t_{n,N}}^{1-\delta}r^{n-1}\left(1-\left(1-\alpha_{n,1-\left(n-1\right)^{-0.5}N^{-\frac{2}{n-1}}}\right)^{\frac{N}{2}}\right)dr=\\
	& |\partial D_{n}|\int_{t_{n,N}}^{1-\delta}r^{n-1}\left(1-\left(1-\frac{2\gamma\left(1+O\left(n^{-1}\right)\right)}{N}e^{-\sqrt{n-1}\left(1+O\left(n^{-0.5}\right)\right)}\right)^{\frac{N}{2}}\right)dr\leq\\
	& |\partial D_{n}|\int_{t_{n,N}}^{1-\delta}r^{n-1}\left(1-e^{^{-\gamma\left(1+O\left(n^{-1}\right)\right)e^{-\sqrt{n-1}\left(1+O\left(n^{-0.5}\right)\right)}}}\right)dr\leq\\
	& C|\partial D_{n}|(1-\delta-t_{n,N})e^{-\gamma\sqrt{n-1}}=|D_{n}|\NN o\left(n^{-0.5}\right).
	\end{aligned}
	\end{equation}
\end{proof}

\section{Proof of Theorem \ref{sec_thm}}
Recall that we want to find an upper bound for  $ \Delta_s\left(Q^b_{n,N}, D_n\right)$, where $ Q^b_{n,N} $ is a polytope in $ \R^n $ with at most $ N $ facets that minimizes the surface area deviation with the Euclidean ball.

For this purpose, choose a polytope $P$ from the random construction that was used in Theorem \ref{main_thm} which satisfies both: 
\[
|D_n \setminus P| \leq \left(\int_{0}^{1} t^{-1}(1-e^{-\ln(2)t})dt + O(n^{-0.5})\right)\NN|D_n|
\]
and
\[
|P \setminus D_n| \leq \left( \int_{0}^{\infty}e^{-\ln(2)e^{t}}dt + O(n^{-0.5})\right)\NN|D_n|.
\]
%

First, we find a lower bound for $ |\partial\left(D_n \cap P\right)|.$
\begin{lemma}{\label{lem1}}
	\be
	|\partial\left(P\cap D_n\right)| \geq  \left(1  - \left(\int_{0}^{1} t^{-1}(1-e^{-\ln(2)t})dt\right)\NN\right)|\partial D_n|.
	\ee
\end{lemma}
\begin{proof}
	 By definition, $ P $ satisfies the inequality
	\[|P\cap D_n| \geq \left(1  -\left(\int_{0}^{1} t^{-1}(1-e^{-\ln(2)t})dt + O(n^{-0.5})\right)\NN\right)|D_n|,\] and by the isoperimetric inequality (Lemma \ref{isoperemetric}) 
	\begin{equation*}
	\begin{aligned}
	|\partial\left(P\cap D_n\right)| &\geq n|P\cap D_n|^{\frac{n-1}{n}}|D_n|^{\frac {{1}}{n}}\\&\geq \left(\left(\int_{0}^{1} t^{-1}(1-e^{-\ln(2)t})dt  + O(n^{-0.5})\right)\NN\right)|\partial D_n|.
	\end{aligned}
	\end{equation*}
	 The lemma follows.
\end{proof}
Finally, we prove an upper bound for $ |\partial\left(D_n\cup P\right)| $.
\begin{lemma}{\label{lem2}}
	\be
	\begin{aligned} & |\partial\left(P\cup D_{n}\right)|\leq \\&\left(1+ \left(\int_{0}^{1} t^{-1}(1-e^{-\ln(2)t})dt + \int_{0}^{\infty}e^{-\ln(2)e^{t}}dt + \frac{1}{2} +O(n^{-0.5})\right)\NN\right)|\partial D_{n}|.
	\end{aligned}
	\ee
\end{lemma}
\begin{proof}
	By the definition of the symmetric volume difference, $ P $ satisfies the inequality
	\begin{equation}{\label{aa}}
	|P \cup D_n| \leq \left(1+ \left(\int_{0}^{1} t^{-1}(1-e^{-\ln(2)t})dt + \int_{0}^{\infty}e^{-\ln(2)e^{t}}dt + O(n^{-0.5})\right)\NN\right)|D_n|.
	\end{equation}
	By volume considerations, we notice that the origin is in the interior of $ P$. Hence, by the cone-volume formula,
	\begin{equation}{\label{bb}}
	\begin{aligned}
	|D_{n}\cup P|&= |\text{conv}(\vec{0},\partial P \cap D^c_{n})| + |\text{conv}(\vec{0},\partial D_{n}\cap P^c)| \\&= \frac{t_{n,N}}{n}|\partial P \cap D^c_{n}|+{\frac {{1}}{n}}|\partial D_{n}\cap P^c|,
	\end{aligned}
	\end{equation}
	where in the last equality we used the fact all the facets have the same height $ t_{n,N}.$  
	Now we use both Eqs. \eqref{aa} and \eqref{bb}  to derive that
	\begin{align*}
	&\frac {t_{n,N}}{n}|\partial\left(P \cup D_n\right)| \leq 
	\\&\left(1+ \left(\int_{0}^{1} t^{-1}(1-e^{-\ln(2)t})dt + \int_{0}^{\infty}e^{-\ln(2)e^{t}}dt + O(n^{-0.5})\right)\NN\right)|D_n|.
	\end{align*}
	Since $ t_{n,N}= {1-\frac 12\left(1+O\left(n^{-0.5}\right)\right)\NN},$ the lemma follows.
\end{proof}
\begin{proof}[Proof of Theorem \ref{sec_thm}]
The theorem now follows by using Lemmas \ref{lem1} and \ref{lem2} and the definition of the surface area deviation:
\be
\begin{aligned}
\Delta_s\left(D_n,P\right)&=|\partial\left(P\cup D_n\right)|-|\partial\left(P\cap D_n\right)|\\&\leq \left(2\int_{0}^{1} t^{-1}(1-e^{-\ln(2)t})dt + \int_{0}^{\infty}e^{-\ln(2)e^{t}}dt + \frac{1}{2} + O(n^{-0.5})\right)\NN|\partial D_n|.
\end{aligned}
\ee
\end{proof}
\section{Proof of Theorem \ref{suprisig}}
Let $ P^b_{n,N} $ be the polytope in $ \R^n $ with at most $ N $ facets that minimizes the symmetric volume difference with the $n$-dimensional Euclidean unit ball. In Theorem 2 of \cite{Lud06}, it was shown that
\begin{align*}
|D_n\setminus P_{n,N}| \geq \frac 1{n}\sum_{i=1}^{N}|F_i\cap D_n|\sqrt{1-(F_i\cap D_n)^{\frac{2}{n-1}}},
\end{align*}
where $F_1,\ldots,F_N$ denote the facets of $ P_{n,N}.$ By Lemma 9 in \cite{Lud06}, each facet of $ P^b_{n,N} $ satisfies
\[
|F_i \cap D_n| = |F_i \cap D^c_n|.
\]
We define $ \sqrt{1-r_i^2}$ to be the height such that $|D_n \cap \{x_1=\sqrt{1-r_i^2}\}| = |F_i|.$ From this definition, we know that $ d(o,F_i) > \sqrt{1-r_i^2} $ and $ |F_i \cap D_n| = \frac{1}{2}|F_i| = r_i^{n-1}|D_{n-1}|.$  Thus

\begin{align}\label{target}
|D_n\setminus P_{n,N}| \geq \frac {|D_{n-1}|}{2n}\sum_{i=1}^{N}r_i^{n-1}\left(1-\sqrt{1-r_{i}^{2}}\right).
\end{align} 
We formulate an optimization problem, whose target function is smaller than the right-hand side of Eq. \eqref{target} and the constraint is the surface area of our polytope,
\[
\min\left\{f\left(r_1,\ldots,r_N\right) :\ |D_{n-1}|\sum_{i=1}^{N}r_{i}^{n-1}=|\partial P^b_{n,N}| \ ,0\leq r_{i}\leq 1  ,\forall i\in1,\ldots,N\right\},
\]
where
\[
f\left(r_1,\ldots,r_N\right) = \frac{|D_{n-1}|}{2n}\sum_{i=1}^{N}r_{i}^{n-1}\left(1-\sqrt{1-r_{i}^{2}}\right).
\]
Using  Lagrange multipliers and the separability of both $ f$ and the constraints, we derive that the minimum is achieved at the point 
\[
r^*_1=\cdots=r^*_N = \left(\frac{|\partial P|}{|D_{n-1}|N}\right)^{\frac 1{n-1}}. 
\]
We conclude that
\begin{equation}
\begin{aligned}\Delta_{v}(P_{n,N}^{b},D_{n}) & \geq f\left(r_{1}^{*},\ldots,r_{N}^{*}\right)=\frac{|D_{n-1}|}{2n}\sum_{i=1}^{N}\frac{|\partial P_{n,N}^{b}|}{N|D_{n-1}|}\left(1-\sqrt{1-\left(\frac{|\partial P_{n,N}^{b}|}{|D_{n-1}|N}\right)^{\frac{2}{n-1}}}\right)\\
& =\frac{|\partial P|}{2n}\left(1-\sqrt{1-\left(\frac{|\partial P_{n,N}^{b}|}{|D_{n-1}|N}\right)^{\frac{2}{n-1}}}\right)\\
& \geq(\frac{1}{2}-c\NN)|D_{n}|\left(1-\sqrt{1-\left(\frac{|\partial P_{n,N}^{b}|}{|D_{n-1}|N}\right)^{\frac{2}{n-1}}}\right)\\
& \geq\bigg(\frac{1}{4}-c\NN+O(n^{2}N^{-\frac{4}{n-1}})\bigg)\NN|D_{n}|,
\end{aligned}
\end{equation}
where we used the isoperimetric inequality (Lemma \ref{isoperemetric}), Theorem \ref{main_thm} (which implies $ |\partial P| \geq (1-c\NN) |\partial D_n|$) and $ \sqrt{1-x} = 1-\frac{1}{2}x+O(x^2). $ Hence, by taking $ N \to \infty $ 
\[
\frac 12\text{ldiv}_{n-1}|\partial D_n|^{1+\frac{2}{n-1}} \geq \frac{1}{4}|D_{n}|
\]
so by Stirling's inequality we obtain $ \text{ldiv}_{n-1} \geq (4\pi e)^{-1} + o(1)$, as desired.
\qed

\section*{ACKNOWLEDGMENTS}
I would like to express my sincerest gratitude to Prof. Bo'az Klartag for the inspiring discussions, and also to Prof. Gideon Schechtman and Dr. Ronen Eldan. 
Also I express my gratitude to my friend Prof. Steven Hoehner and Ms. Anna Mendelman for editing the content of this paper.

\section{Technical lemmas and loose ends}{\label{Techandloose}}
Recall that
\[ 
\alpha_{n,r} = \frac{2|D_{n-1}|}{|D_n|}\left(\int_{\frac{t_{n,N}}{r}}^{1}\left(1-x^2\right)^{\frac{n-1}{2}}dx  +\frac {t_{n,N}}{nr} \left(1-\frac{t_{n,N}^2}{r^2}\right)^{\frac{n-1}{2}}\right) 
\]
where $t_{n,N} = \sqrt{1 - \left(\frac{\gamma |\partial D_n|}{N|D_{n-1}|}\right)^{\frac 2{n-1}}}$.  The integral is the volume of the cap, and the second term is the volume of the cone whose common base is $ \{x\in\R^n: x_1 = \frac {t_{n,N}}r\} \cap D_n.$ 
When $N\geq n^n$, $ t_{n,N} $ is very close to 1. When $ r $ is close to 1, the volume of the cone is significantly larger than the volume of the cap. The following lemma formalizes this.
\begin{lemma}{\label{sub_main_lemma}}
	Assume that $ r \in [1-\frac{\NN}{\sqrt{n-1}}, 1+\frac{\NN}{\sqrt{n-1}}]$. Then for all $N\geq n^n$,
	\be
	\int_{\frac{t_{n,N}}{r}}^{1}\left(1-x^2\right)^{\frac{n-1}{2}}dx \leq \frac{C}{n^2}\left(\frac {t_{n,N}}{nr} \left(1-\frac{t_{n,N}^2}{r^2}\right)^{\frac{n-1}{2}}\right).
	\ee
\end{lemma}
\begin{proof} Observe that  $ \NN = O(n^{-2}),$ which implies that  $ \frac {t_{n,N}}r = 1-O\left(n^{-2}\right)$. Hence, 
	\begin{equation}
\begin{aligned}\int_{\frac{t_{n,N}}{r}}^{1}\left(1-x^{2}\right)^{\frac{n-1}{2}}dx & =\int_{\frac{t_{n,N}}{r}}^{1}(1+x)^{\frac{n-1}{2}}\left(1-x\right)^{\frac{n-1}{2}}dx\\
& \leq2^{\frac{n-1}{2}}\int_{t_{n,N}}^{1-\frac{t_{n,N}}{r}}x^{\frac{n-1}{2}}dx=\frac{2^{\frac{n+1}{2}}}{n-1}\left(1-\frac{t_{n,N}}{r}\right)^{\frac{n+1}{2}}\\
& \leq\frac{C}{n^{3}}2^{\frac{n+1}{2}}\left(1-\frac{t_{n,N}}{r}\right)^{\frac{n-1}{2}}\leq\frac{C}{2n^{3}}\left(1+\frac{t_{n,N}}{r}\right)^{\frac{n-1}{2}}\left(1-\frac{t_{n,N}}{r}\right)^{\frac{n-1}{2}}\\
& \leq\frac{C}{n^{2}}\frac{t_{n,N}}{nr}\left(1-\left(\frac{t_{n,N}}{r}\right)^{2}\right)^{\frac{n-1}{2}}.
\end{aligned}
	\end{equation}
\end{proof}	

Now we can complete all the missing details from the proof of Theorem \ref{main_thm}.
First we prove Lemma \ref{main_lemma}.
\begin{lemmaa}
	Assume that $ r \in [1-\frac{\NN}{\sqrt{n-1}}, 1+\frac{\NN}{\sqrt{n-1}}].$ Then it holds that
	\begin{equation}
	\alpha_{n,r}=\frac{2\gamma\left(1+O\left(n^{-1}\right)\right)}{N}e^{\left(n-1\right)\left(r-1\right)N^{\frac{2}{n-1}}\left(1+O\left(n^{-0.5}\right)\right)}
	\end{equation}
	\begin{proof} Using Lemma \ref{sub_main_lemma} and the fact that both $ t_{n,N}$ and $r $ are of the order $ 1-O\left(n^{-2}\right)$, we derive that
		\begin{equation*}
\begin{aligned}\alpha_{n,r} & =\left(1+O\left(n^{-1}\right)\right)\frac{2|D_{n-1}|}{|D_{n}|}\left(\frac{t_{n,N}}{nr}\right)\left(1-\frac{t_{n,N}^{2}}{\left(1+\left(r-1\right)\right)^{2}}\right)^{\frac{n-1}{2}}\\
& =\left(1+O\left(n^{-1}\right)\right)\frac{2|D_{n-1}|}{|\partial D_{n}|}\frac{1}{\left(1+\left(r-1\right)\right)^{n-1}}\left(\left(1+\left(r-1\right)\right)^{2}-t_{n,N}^{2}\right)^{\frac{n-1}{2}}\\
& =\left(1+O\left(n^{-1}\right)\right)\frac{2|D_{n-1}|}{|\partial D_{n}|}\left(1-t_{n,N}^{2}+2\left(r-1\right)+\left(r-1\right)^{2}\right)^{\frac{n-1}{2}}\\
& =\left(1+O\left(n^{-1}\right)\right)\frac{2|D_{n-1}|}{|\partial D_{n}|}\left(\left(\frac{\gamma|\partial D_{n}|}{|D_{n-1}|N}\right)^{\frac{2}{n-1}}+2\left(r-1\right)+\left(r-1\right)^{2}\right)^{\frac{n-1}{2}}\\
& =\frac{2\gamma\left(1+O\left(n^{-1}\right)\right)}{N}\left(1+2\left(r-1\right)N^{\frac{2}{n-1}}(1+O(\frac{ln\left(n\right)}{n}))\right)^{\frac{n-1}{2}}\\
& =\frac{2\gamma\left(1+O\left(n^{-1}\right)\right)}{N}e^{\left(n-1\right)\left(r-1\right)N^{\frac{2}{n-1}}\left(1+O\left(n^{-0.5}\right)\right)}.
\end{aligned}
		\end{equation*}
	\end{proof}
\end{lemmaa}
The following is proof the of Lemma \ref{tech_lemma}.
\begin{lemmaa}
For all $r \geq n^2$, it holds that
	\be
	\alpha_{n,r} \geq 1-\frac{C\sqrt{n}}{r}	.
	\ee
\end{lemmaa}
\begin{proof} 
We have
\begin{align*}\alpha_{n,r} & =\frac{2|D_{n-1}|}{|D_{n}|}\left(\int_{\frac{t_{n,N}}{r}}^{1}\left(1-x^{2}\right)^{\frac{n-1}{2}}dx+\frac{t_{n,N}}{nr}\left(1-\frac{t^{2}}{r^{2}}\right)^{\frac{n-1}{2}}\right)\\
& \geq\frac{2|D_{n-1}|}{|D_{n}|}\int_{\frac{t_{n,N}}{r}}^{1}\left(1-x^{2}\right)^{\frac{n-1}{2}}dx \geq\frac{2|D_{n-1}|}{|D_{n}|}\int_{\frac{t_{n,N}}{r}}^{1}\left(1-x^{2}\right)^{\frac{n-1}{2}}dx\\
& =1-2\frac{|D_{n-1}|}{|D_{n}|}\int_{0}^{\frac{t_{n,N}}{r}}\left(1-x^{2}\right)^{\frac{n-1}{2}}dx,
\end{align*}
where in the last equality we used the fact that  $ |D_{n-1}|\int_{0}^{1}\left(1-x^{2}\right)^{\frac{n-1}{2}}dx= \frac{|D_n|}{2}$. Continuing from the previous line, we obtain
\begin{align*}\quad\quad\quad\quad\quad\quad & \geq1-2\frac{|D_{n-1}|}{|D_{n}|}\int_{0}^{\frac{1}{r}}\left(1-x^{2}\right)^{\frac{n-1}{2}}dx\geq1-c\sqrt{n}\int_{0}^{\frac{1}{r}}\left(1-x^{2}\right)^{\frac{n-1}{2}}dx\\
& \geq1-c_{1}\sqrt{n}\int_{0}^{\frac{1}{r}}\left(1-x\right)^{\frac{n-1}{2}}dx\\
& \geq1-\frac{c\sqrt{n}}{n}\left(1-\left(1-\frac{1}{r}\right)^{\frac{n+1}{2}}\right)\\
& \geq1-\frac{c}{\sqrt{n}}\left(1-\left(1-\frac{n+1}{2r}+\ldots\right)\right)\\
& \geq1-\frac{c\sqrt{n}}{r}.
\end{align*}
\end{proof}
\subsection*{Sketch of the proof of  Remark \ref{niceremark}}
 We give short proofs of the  modifications needed so that Theorem \ref{main_thm} holds when the random polytope has at most $ 10^n \leq N \leq n^n$ facets.   
For this purpose, we modify  Lemmas \ref{firstpartlem} and \ref{lemma_2} so that they will hold when  $ 10^n \leq N \leq n^n.$ For both the aforementioned lemmas, we need to estimate the volume of a spherical cap with height $ h< 1 $. For this purpose, we shall use the following integration by parts identity:
\begin{equation}\label{IntByParts}
	\begin{aligned}
\int_{a}^{b}e^{ng(x)}dx &= \frac{1}{n}\bigg[ \frac{1}{g'(b)}e^{ng(b)} - \frac{1}{g'(a)}e^{ng(a)} \bigg]- \frac{1}{n}\int_{a}^{b}\frac{d}{dx}\left(\frac{1}{g'(x)}\right)e^{ng(x)}dx.
\end{aligned}
\end{equation}	
\begin{lemma}{\label{laplaceaprox}}
	Let $ a_n \in (\frac{2}{3},1)$ be a number that may depend on the dimension $ n $. Then the following holds:
	\be
	\begin{aligned}
	\int_{a_n}^{1}\left(1-x^{2}\right)^{\frac{n-1}{2}}dx = \frac{\left(1-a_n^{2}\right)^{\frac{n+1}{2}}}{a_n\left(n-1\right)}+O\left(\frac{\int_{a_n}^{1}\left(1-x^{2}\right)^{\frac{n-1}{2}}dx}{n}\right).
\end{aligned}
	\ee
\end{lemma}
\begin{proof}
	Let $ \varepsilon < \frac{1-a_n}{2}$. Then
\begin{align*}\int_{a_{n}}^{1-\varepsilon}\left(1-x^{2}\right)^{\frac{n-1}{2}}dx & =\int_{a_{n}}^{1-\varepsilon}e^{\frac{n-1}{2}\ln(1-x^{2})}dx\\
& =\frac{2}{n-1}\bigg[-\frac{1-(1-\varepsilon)^{2}}{2(1-\varepsilon)}(1-(1-\varepsilon)^{2}){}^{\frac{n-1}{2}}\\
& \,\,\,+\frac{1-a_{n}^{2}}{2a_{n}}(1-a_{n}^{2}){}^{\frac{n-1}{2}}\bigg]-\frac{2}{n-1}\int_{a_{n}}^{1-\varepsilon}\frac{(1-x^{2})}{2x}\left(1-x^{2}\right)^{\frac{n-1}{2}}dx\\
& \leq\frac{2}{n-1}\bigg[-\frac{(1-(1-\varepsilon)^{2})^{\frac{n+1}{2}}}{2(1-\varepsilon)}+\frac{1-a_{n}^{2}}{2a_{n}}(1-a_{n}^{2}){}^{\frac{n-1}{2}}\bigg]+\\
& \quad\quad\frac{C}{n}\int_{a_{n}}^{1-\varepsilon}\left(1-x^{2}\right)^{\frac{n-1}{2}}dx,
\end{align*}
	where the second equality follows from Eq. \eqref{IntByParts}. Taking the limit of both sides of the previous inequality as $\varepsilon\to 0$ yields the lemma. 
\end{proof}
Now we show how to modify the proof of Lemma \ref{lemma_2}; Lemma \ref{firstpartlem} can be obtained by similar modifications.
For this purpose, we need to derive a lower bound for $ \alpha_{n,r}$.  First, we  show that the volume of aforementioned cone is larger than the volume of the spherical cap.
\begin{lemma}{\label{sub_main_lemma_2}}
	Assume that $ r \in [1, 1+\frac{\NN}{\sqrt{n-1}}]$ and $ 10^{n}\leq N\leq n^{n} $. When the dimension is sufficiently large,  it holds that
	\begin{equation*}
	\int_{\frac{t_{n,N}}{r}}^{1}\left(1-x^2\right)^{\frac{n-1}{2}}dx \leq \frac{1}{100}\frac {t_{n,N}}{nr} \left(1-\frac{t_{n,N}^2}{r^2}\right)^{\frac{n-1}{2}}.
	\end{equation*}
\end{lemma}
\begin{proof}
	Applying Lemma \ref{laplaceaprox} with $ a_n = \frac {t_{n,N}}r$ yields
\begin{align*}\int_{\frac{t_{n,N}}{r}}^{1}\left(1-x^{2}\right)^{\frac{n-1}{2}}dx & =\frac{1}{n-1}\frac{r}{t_{n,N}}\left(1-\frac{t_{n,N}^{2}}{r^{2}}\right)^{\frac{n+1}{2}}+O\left(n^{-1}\int_{\frac{t_{n,N}}{r}}^{1}\left(1-x^{2}\right)^{\frac{n-1}{2}}dx\right)\\
& \leq\frac{1}{100}\frac{t_{n,N}}{nr}\left(1-\frac{t_{n,N}^{2}}{r^{2}}\right)^{\frac{n-1}{2}}.
\end{align*}
\end{proof}
Now we modify Lemma \ref{lemma_2}.
Using  Lemma \ref{sub_main_lemma_2}, one can repeat the proof of Lemma \ref{main_lemma} to derive the following
\begin{lemma}[Modification of Lemma \ref{main_lemma}]{\label{alphnanrmodified}}
	Assume that $ r \in [1, 1+\frac{\NN}{\sqrt{n-1}}]$ and $ 10^{n}\leq N\leq n^{n} $. Then 
	\begin{equation*}
	\frac{2\gamma\left(1-\frac{1}{25}+ O\left(n^{-1}\right)\right)}{N}e^{\left(n-1\right)\left(r-1\right)N^{\frac{2}{n-1}}\left(1+O\left(n^{-0.5}\right)\right)}\leq\alpha_{n,r}.
	\end{equation*}
\end{lemma}
Finally we show how to modify Lemma \ref{lemma_2}.
\begin{lemma}[Modification of Lemma \ref{lemma_2}]
	\begin{equation*}
	\E[|P \setminus D_n|]  \leq \left(\frac{\int_{0}^{\infty}e^{-\ln(2)e^{t}}dt}{1-\frac{1}{20}}+O\left(n^{-0.5}\right)\right)N^{-\frac{2}{n-1}}|D_n|.
	\end{equation*}
\end{lemma}
\begin{proof}
	We define $ \delta = \min\{\NN(n-1)^{-0.5},(100n)^{-1}\}$ and split $\E[|P \setminus D_n|]  $ into three parts:
	\begin{equation*}
\begin{aligned}
\E[|P\setminus D_n|]=|\partial D_{n}| & \bigg(\int_{1}^{1+\delta}r^{n-1}\left(1-\alpha_{n,r}\right)^{\frac{N}{2}}dr+\int_{1+\delta}^{n^{2}}r^{n-1}\left(1-\alpha_{n,r}\right)^{\frac{N}{2}}dr\\
& \,\,+\int_{n^{2}}^{\infty}r^{n-1}\left(1-\alpha_{n,r}\right)^{\frac{N}{2}}dr\bigg).
\end{aligned}
	\end{equation*}
	We handle the third integral in the same way as in the Lemma \ref{lemma_2}. Moreover, the second integral is negligible:
	\begin{equation*}
\begin{aligned} & |\partial D_{n}|\int_{1+\delta}^{n^{2}}r^{n-1}\left(1-\alpha_{n,r}\right)^{\frac{N}{2}}dr\leq\\
& |\partial D_{n}|\int_{1+\delta}^{n^{2}}r^{n-1}\left(1-\alpha_{n,1+\left(n-1\right)^{-0.5}N^{-\frac{2}{n-1}}}\right)^{\frac{N}{2}}dr=\\
& |\partial D_{n}|\int_{1+\delta}^{n^{2}}r^{n-1}\left(1-\frac{2\ln(2)\left(1-\frac{1}{25}+O\left(n^{-1}\right)\right)}{N}e^{\sqrt{n-1}\left(1+O\left(n^{-0.5}\right)\right)}\right)^{\frac{N}{2}}dr\leq\\
& |\partial D_{n}|\int_{1+\delta}^{n^{2}}r^{n-1}e^{-\ln(2)\left(1-\frac{1}{25}+O\left(n^{-1}\right)\right)e^{\sqrt{n-1}\left(1+O\left(n^{-0.5}\right)\right)}}dr\leq\\
& |\partial D_{n}|e^{-\ln(2)\left(1-\frac{1}{25}+O\left(n^{-1}\right)\right)e^{\sqrt{n-1}\left(1+O\left(n^{-0.5}\right)\right)}}\int_{t_{n,N}}^{n^{2}}r^{n-1}dr\leq\\
& C|D_{n}|n^{n^{2}}e^{-\ln(2)\left(1-\frac{1}{25}+O\left(n^{-1}\right)\right)e^{\sqrt{n-1}\left(1+O\left(n^{-0.5}\right)\right)}}=o(n^{-3})|D_{n}|\NN.
\end{aligned}
	\end{equation*}
	Finally,  using the lower bound for $ \alpha_{n,r}$ that was proven in Lemma \ref{alphnanrmodified}, we can handle the first integral as we did in Lemma \ref{lemma_2} to derive that
	\begin{align*}
	\E[|P \setminus D_n|]  \leq \left(\frac{\int_{0}^{\infty}e^{-\ln(2)e^{t}}dt}{1-\frac{1}{20}}+O\left(n^{-0.5}\right)\right)N^{-\frac{2}{n-1}}|D_n|.
	\end{align*}
\end{proof}
 	\bibliographystyle{plainnat}	 	
\bibliography{Final}
\end{document}